\documentclass[11pt,reqno]{amsart}

 \usepackage{amsmath,amsthm,amsfonts}
\usepackage{latexsym,mathabx,txfonts}
\usepackage{amssymb}

\newcommand{\R}{\mathbb{R}}

\newcommand{\Z}{\mathbb{Z}}

\renewcommand{\hat}{\widehat}

\numberwithin{equation}{section}

\newtheorem{thm}{Theorem}[section]

\newtheorem{lem}[thm]{Lemma}
\newtheorem{prop}[thm]{Proposition}

\theoremstyle{remark}
\newtheorem{rem}{Remark}[section]

\newcommand{\Del}[1]{}

\begin{document}

\title{Small data global regularity for half-wave maps}
\author{Joachim Krieger and Yannick Sire}

\subjclass{35L05, 35B40}

\keywords{wave equation, fractional wave maps}

\begin{abstract}
We formulate the half-wave maps problem with target $S^2$ and prove global regularity in sufficiently high spatial dimensions for a class of small critical data in  Besov spaces. 
\end{abstract}

\maketitle

\section{The problem}
Let $u: \R^{n+1}\rightarrow S^2\hookrightarrow \R^3$ smooth, and assume that it converges to some $p\in S^2$ at spatial infinity. Further, assume that on each fixed time slice $\nabla_{t,x} u\in L^p(\R^n)$ for some $p\in (1,\infty)$. Denote by $\times$ the standard vectorial product in three dimensions. We call this a {\it{fractional wave map}}, provided it satisfies the following relation: 
\begin{equation}\label{eq:fracwm}
\partial_t u  = u \times (-\triangle)^{\frac{1}{2}}u
\end{equation}
Here we define the operator $(-\triangle)^{\frac{1}{2}}$ via 
\[
(-\triangle)^{\frac{1}{2}}u = -\sum_{j=1}^n (-\triangle)^{-\frac{1}{2}}\partial_j (\partial_j u),
\]
a specification necessary on account of the fact that $u$ does not vanish at infinity, but instead approaches some $p\in S^2$, while $\nabla u$ does vanish at infinity. In fact, the expression $(-\triangle)^{\frac{1}{2}}u$ under our current definition is then well-defined since $\nabla_{t,x}u(t, \cdot)\in L^p(\R^n)$ for some $p\in (1,\infty)$, for all $t$. 
\\

We note that the model \eqref{eq:fracwm} appears formally related to the much-studied Schroe-dinger Maps problem which can be written in the form
\[
\partial_t u = u\times \triangle u,
\]
and moreover, we shall see shortly that \eqref{eq:fracwm} also appears closely related to the classical Wave Maps problem with target $S^2$. We also note that we have a formally conserved quantity 
\begin{equation}\label{eq:enconserv}
E(t): = \int_{\R^n}\big|(-\triangle)^{\frac{1}{4}}u\big|^2\,dx,
\end{equation}
where we let $(-\triangle)^{\frac{1}{4}}u =  -\sum_{j=1}^n(-\triangle)^{-\frac{3}{4}}\partial_j(\partial_j u)$. Such kind of quantities have been considered in the works of Da Lio and Rivi\`ere in the study of fractional harmonic maps (see for instance  \cite{DLR1,DLR2,DaLio1}). We also note that on account of the results on fractional harmonic maps previously mentioned, this model moreover displays a very rich class of static solutions (see also \cite{MS}).

On the other hand, \eqref{eq:fracwm} scales just like Wave Maps, which means that in all dimensions $n\geq 2$ the problem \eqref{eq:fracwm} is formally supercritical. 
\\

We formulated the model \eqref{eq:fracwm} as a toy model in 2014, but have since learned from E. Lenzmann \footnote{The name of half-wave map was suggested by E. Lenzmann.} that it already exists in the physics literature. We learned from E. Lenzmann that the half-wave map equation arises as the continuum version of the  so-called integrable spin Calogero-Moser systems, which in turn comes from the completely integrable quantum spin systems called Haldane-Shastry systems\footnote{E. Lenzmann provided us with the following references \cite{haldane,shasty,hikami,blom} and we refer to his work for an account on the passage from the physics to the mathematical model.}.  Recent work by Lenzmann and Schikorra \cite{Le2} completely classifies the travelling wave solutions for this model in the critical case $n = 1$. 
\\

In the present note, our goal is to approach the issue of global solutions corresponding to small data, attempting to parallel the developments in \cite{Tat3}, \cite{T1}. We will see that \eqref{eq:fracwm} can be reformulated as a nonlinear wave type equation of the schematic form 
\begin{equation}\label{eq:roughschematic}
\Box u = F(u)\nabla_{t,x}u\cdot \nabla_{t,x}u,
\end{equation}
although this is an oversimplification as the true underlying wave equation displays non local-expressions. It has been known now for a while, see \cite{Ste},  that \eqref{eq:roughschematic} admits global solutions corresponding to initial data of small critical, i. e. scaling invariant, Besov $\dot{B}^{\frac{n}{2},1}_2$ norm, provided one restricts oneself to spatial dimensions $n\geq 6$, and that passing to lower dimensions appears to require some sort of null-structure. 
Here, we show that \eqref{eq:fracwm} does have enough of an intrinsic null-structure to allow for the following 
\begin{thm}\label{thm:Main} Let $n\geq 5$. Let $u[0] = (u(0, \cdot), u_t(0, \cdot)) = (u_0, u_1): \R^{n}\longrightarrow S^2\times TS^2$ a smooth data pair with $u_1 = u_0\times (-\triangle)^{\frac12}u_0$, and such that $u_0$ is constant outside of a compact subset of $\R^n$ (this condition in particular ensures that $(-\triangle)^{\frac12}u_0$ is well-defined). Also, assume the smallness condition 
\[
\big\|u[0]\big\|_{\dot{B}^{\frac{n}{2},1}_2\times \dot{B}^{\frac{n}{2}-1,1}_2}<\epsilon
\]
where $\epsilon\ll 1$ sufficiently small. Then problem \eqref{eq:fracwm} admits a global smooth solution. 
\end{thm}

To prove this theorem, we shall have to reformulate \eqref{eq:fracwm} as a wave equation, which we do next. 
\begin{rem} We note that the restriction $n\geq 5$ comes from the fact that we use the $L_t^2 L_x^4$-Strichartz estimate, which is not available in spatial dimension $n = 4$. However, it is quite likely that this can be circumvented, and that the structures exhibited in this paper suffice to push the result to $n = 4$. However, both the issue of passing to the critical space $\dot{H}^{\frac{n}{2}}$, as well as going to lower spatial dimensions $n\leq 3$, appear non-trivial, as there are novel trilinear terms which no longer seem to have the same strong null-structure as the leading term coming from the Wave Maps equation. 

\end{rem}

\section{Passage to a wave equation}

Departing from \eqref{eq:fracwm}, we compute 
\begin{align*}
\partial_t^2 u &= \partial_t u \times (-\triangle)^{\frac{1}{2}}u + u \times (-\triangle)^{\frac{1}{2}}\partial_t u\\
&=(u\times (-\triangle)^{\frac{1}{2}}u)\times  (-\triangle)^{\frac{1}{2}}u + u \times (-\triangle)^{\frac{1}{2}}(u \times (-\triangle)^{\frac{1}{2}}u)
\end{align*}
Then using the basic formula $a\times (b\times c) = b(a\cdot c) - c(a\cdot b)$, we re-write the first term on the right as 
\begin{align*}
(u\times (-\triangle)^{\frac{1}{2}}u)\times  (-\triangle)^{\frac{1}{2}}u  =  -u ((-\triangle)^{\frac{1}{2}}u\cdot(-\triangle)^{\frac{1}{2}}u) + (-\triangle)^{\frac{1}{2}}u(u\cdot (-\triangle)^{\frac{1}{2}}u)
\end{align*}
For the second term on the right above, introducing a commutator term, we write it in the form 
\begin{align*}
&u \times (-\triangle)^{\frac{1}{2}}(u \times (-\triangle)^{\frac{1}{2}}u)\\
&=u \times (-\triangle)^{\frac{1}{2}}(u \times (-\triangle)^{\frac{1}{2}}u) - u \times(u \times (-\triangle)u)\\
&+u \times(u \times (-\triangle)u)\\
&=u \times (-\triangle)^{\frac{1}{2}}(u \times (-\triangle)^{\frac{1}{2}}u) - u \times(u \times (-\triangle)u)\\
&+u(u\cdot (-\triangle)u) +\triangle u
\end{align*}
Using the fact that $u\cdot u = 1$, whence 
\[
u\cdot \triangle u + \nabla u\cdot\nabla u = 0, 
\]
we arrive at the equation 
\begin{align*}
(\partial_t^2 - \triangle) u &= -u ((-\triangle)^{\frac{1}{2}}u\cdot(-\triangle)^{\frac{1}{2}}u) + (-\triangle)^{\frac{1}{2}}u(u\cdot (-\triangle)^{\frac{1}{2}}u)\\
&+u \times (-\triangle)^{\frac{1}{2}}(u \times (-\triangle)^{\frac{1}{2}}u) - u \times(u \times (-\triangle)u)\\
&+u(\nabla u\cdot\nabla u)
\end{align*}
Carefully note that $\nabla u$ here only involves the spatial derivatives. 
In order to make this appear closer to the Wave Maps equation and introduce better null-structure, we have to also make the time derivatives visible on the right hand side, for which the first line on the right hand side is pivotal. In fact, we get  
\begin{align*}
&\big(-u ((-\triangle)^{\frac{1}{2}}u\cdot(-\triangle)^{\frac{1}{2}}u) + (-\triangle)^{\frac{1}{2}}u(u\cdot (-\triangle)^{\frac{1}{2}}u)\big)\cdot u\\
&= - \big|u\times (-\triangle)^{\frac{1}{2}}u\big|^2 = -|\partial_tu|^2,
\end{align*}
and so the equation becomes 
\begin{equation}\label{eq:fracWMfinal}\begin{split}
(\partial_t^2 - \triangle) u &= u(\nabla u\cdot\nabla u - \partial_t u\cdot\partial_t u)\\
&+\Pi_{u_{\perp}}\big((-\triangle)^{\frac{1}{2}}u\big)(u\cdot (-\triangle)^{\frac{1}{2}}u)\\
&+u \times (-\triangle)^{\frac{1}{2}}(u \times (-\triangle)^{\frac{1}{2}}u) - u \times(u \times (-\triangle)u)
\end{split}\end{equation}
where $\Pi_{u_{\perp}}$ denotes projection onto the orthogonal complement of $u$. 
Thus we see that formally the nonlinearity involves the precise wave maps source term, as well as two error terms, which formally behave like 
\[
u\nabla u\cdot\nabla u
\]

\section{Technical preliminaries}

Our main tools shall be the classical Strichartz estimates, combined with some $X^{s,b}$-space technology. Specifically, we let $P_k$, $k\in Z$, be standard Littlewood-Paley multipliers on $\R^n$, and furthermore, we denote by $Q_j$, $j\in \Z$, multipliers which localise a space-time function $F(t, x)$ to dyadic distance $\sim 2^j$ from the light cone $\big|\tau\big| = \big|\xi\big|$ on the Fourier side. Specifically, letting $\tilde{F}(\tau, \xi)$ denote the space time Fourier transform of $F$, and letting $\chi\in C^\infty_0(\R_+)$ a smooth cutoff  satisfying 
\[
\sum_{k\in Z}\chi(\frac{x}{2^k}) = 1\,\forall x\in \R_+, 
\]
we set 
\[
\widetilde{Q_j F} = \chi\big(\frac{\big||\tau| - |\xi|\big|}{2^j}\big)\tilde{F}(\tau, \xi). 
\]
Here, $\tilde{F}(\tau, \xi)$ denotes the space-time Fourier transform. Using these ingredients one can then define the following norms: 
\[
\big\|u\big\|_{\dot{X}^{\frac{n}{2},\frac12,\infty}}: = \sup_{j\in Z}2^{\frac{j}{2}}\big\|\nabla_x^{\frac{n}{2}}Q_j u\big\|_{L_{t,x}^2},\,\big\|F\big\|_{\dot{X}^{\frac{n}{2}-1, -\frac12, 1}}: = \sum_{j\in Z}2^{-\frac{j}{2}}\big\|\nabla_x^{\frac{n}{2}-1}Q_j F\big\|_{L_{t,x}^2}
\]
In addition to these, we rely on the classical Strichartz norms, which are the mixed type Lebesgue norms $\big\|\cdot\big\|_{L_t^p L_x^q}$, $\frac1p + \frac{n-1}{2q}\leq \frac{n-1}{4}$, $p\geq 2$, where we shall always restrict to $n\geq 5$. Call such pairs $(p,q)$ admissible.
We can now define a norm controlling our solutions as follows: 
\begin{equation}\label{eq:Snorm}
\big\|u\big\|_{S}: = \sum_{k\in Z}\sup_{(p,q)\,\text{admissible}}2^{(\frac1p + \frac{n}{q}-1)k}\big\|\nabla_{t,x}P_ku\big\|_{L_t^p L_x^q} + \big\|\nabla_{t,x}P_k u\big\|_{\dot{X}^{\frac{n}{2}-1,\frac12,\infty}} =:\sum_{k\in Z}\big\|P_k u\big\|_{S_k}.
\end{equation}
We also introduce 
\begin{equation}\label{eq:Nnorm}
\big\|F\big\|_{N}: = \sum_{k\in Z}\big\|P_kF\big\|_{L_t^1\dot{H}^{\frac{n}{2}-1} + \dot{X}^{\frac{n}{2}-1,-\frac12,1}}
\end{equation}

Then the following inequality is by now completely standard, see e. g. \cite{Kri}, \cite{T1}, \cite{Tat3}: 
\begin{equation}\label{eq:energyinequality}
\big\| u\big\|_S\lesssim \big\|u[0]\big\|_{\dot{H}^{\frac{n}{2}}\times \dot{H}^{\frac{n}{2}-1}} + \big\|\Box u\big\|_{N}. 
\end{equation}

In order to deal with the nonlocal expressions such as $(-\triangle)^{\frac{1}{2}}(u \times (-\triangle)^{\frac{1}{2}}u)$, the following simple lemma shall be useful: 

\begin{lem}\label{lem:multilinestimates} Consider the following bilinear expression (where $\chi_{k_j}(\cdot)$ smoothly localises to the annulus $|\xi|\sim 2^{k_j}$)
\[
F(u, v)(x): = \int_{\R^{n}}\int_{\R^{n}}m(\xi, \eta)e^{ix\cdot(\xi+\eta)}\chi_{k_1}(\xi)\hat{u}(\xi)\chi_{k_2}(\eta)\hat{v}(\eta)d\xi d\eta
\]
where the multiplier $m(\xi, \eta)$ is $C^\infty$ with respect to the coordinates on the support of $\chi_{k_1}(\xi)\cdot \chi_{k_2}(\eta)$, and satisfies a point wise bounds 
\[
\big|m(\xi, \eta)\big|\leq \gamma\lesssim 1,\, \big|(2^{k_1}\nabla_{\xi})^i(2^{k_2}\nabla_{\eta})^jm(\xi, \eta)\big|\lesssim_{i,j} 1,\,\forall i,,j.  
\]
Then if $\big\|\cdot\big\|_{Z}, \big\|\cdot\big\|_{Y}, \big\|\cdot\big\|_{X}$ are translation invariant norms with the property that 
\[
\big\|u\cdot v\big\|_{Z}\leq \big\|u\big\|_{X}\cdot\big\|v\big\|_{Y}, 
\]
then it follows that 
\[
\big\|F(u, v)\big\|_{Z}\lesssim\gamma^{(1-)}\big\|P_{k_1}u\big\|_{X}\big\|P_{k_2}v\big\|_{Y}
\]
where the implied constant only depends on the size of finitely many derivatives of $m$. 
\end{lem}
\begin{proof} This follows by Fourier expansion of the multiplier $m(\xi, \eta)$: write
\begin{align*}
m(\xi, \eta)\chi_{k_1}(\xi)\chi_{k_2}(\eta) = \sum_{m\in Z^n}\sum_{p\in Z^n}a_{m p}e^{i(2^{-k_1} m\cdot\xi + 2^{-k_2}p\cdot\eta)}
\end{align*}
where we have 
\begin{align*}
&\big|a_{m_1 m_2}\big|\\
&\leq (2^{-k_1}|m| + 2^{-k_2}|p|)^{-Mn}\big\|\nabla_{\xi, \eta}^{Mn}\big[m(\xi, \eta)\chi_{k_1}(\xi)\chi_{k_2}\big]\big\|_{L_{\xi, \eta}^{\infty}}\\
&\lesssim_{M,n}[|m| + |p|]^{-Mn}
\end{align*}
while we also get the trivial bound $\big|a_{m_1 m_2}\big|\big|\leq \gamma$. It follows that 
\begin{align*}
&F(u, v)(x) = \\&\sum_{\substack{m,p\in Z^n\\ |m|+ |p|<\gamma^{-\frac{1}{nM}}}} a_{m p}\int_{\R^{n}}\int_{\R^{n}}\hat{u}(\xi)\hat{v}(\eta)e^{i([2^{-k_1} m+x]\cdot\xi + [2^{-k_2}p+x]\cdot\eta)}
\,d\xi d\eta\\
& + \sum_{\substack{m,p\in Z^n\\ |m|+ |p|\geq \gamma^{-\frac{1}{nM}}}} a_{m p}\int_{\R^{n}}\int_{\R^{n}}\hat{u}(\xi)\hat{v}(\eta)e^{i([2^{-k_1} m+x]\cdot\xi + [2^{-k_2}p+x]\cdot\eta)}
\,d\xi d\eta\\
\end{align*}
and so 
\begin{align*}
\big\|F(u, v)\big\|_{Z}\lesssim &\big\|u\big\|_{X}\big\|v\big\|_{Y}[\sum_{\substack{m,p\in Z^n\\ |m|+ |p|<\gamma^{-\frac{1}{nM}}}}\gamma + \sum_{\substack{m,p\in Z^n\\ |m|+ |p|\geq \gamma^{-\frac{1}{nM}}}}(|m| + |p|)^{-Mn}]\\
&\lesssim  \big\|u\big\|_{X}\big\|v\big\|_{Y}\gamma^{1-\frac{1}{M}}
\end{align*}
Here the constant $M$ may be chosen arbitrarily large (with implied constant depending on $M$). 
\end{proof}

\section{Multilinear estimates}

Here we gather the multilinear estimates which allow us to obtain a solution for \eqref{eq:fracWMfinal} by means of a suitable iteration scheme: 
\begin{prop}\label{prop:multilinear} Assume that $u$ takes values in $S^2$ and converges to $p\in S^2$ at spatial infinity. Then using the norms $\big\|\cdot\big\|_{S}, \big\|\cdot\big\|_{N}$ introduced in the previous section, we have the bounds 
\begin{equation}\label{eq:multilin1}
\big\|P_k\big[u(\nabla u\cdot\nabla u - \partial_t u\cdot\partial_t u)\big]\big\|_{N}\lesssim (1+\big\|u\big\|_{S})\big\|u\big\|_{S}\big(\sum_{k_1\in Z}2^{-\sigma|k-k_1|}\big\|P_{k_1}u\big\|_{S_{k_1}}\big)
\end{equation}
Furthermore, if $\tilde{u}$ maps into a small neighbourhood of $S^2$, we have the similar bound 
\begin{equation}\label{eq:multilin2}
\big\|P_k\big(\Pi_{\tilde{u}_{\perp}}\big((-\triangle)^{\frac{1}{2}}u\big)(u\cdot (-\triangle)^{\frac{1}{2}}u)\big)\big\|_{N}\lesssim \prod_{v = u,\tilde{u}} (1+\big\|v\big\|_{S})\big\|u\big\|_{S}\big(\sum_{k_1\in Z}2^{-\sigma|k-k_1|}\big\|P_{k_1}u\big\|_{S_{k_1}}\big)
\end{equation}
as well as 
\begin{equation}\label{eq:multilin3}\begin{split}
&\big\|P_k\big(\Pi_{\tilde{u}_{\perp}}\big[u \times (-\triangle)^{\frac{1}{2}}(u \times (-\triangle)^{\frac{1}{2}}u) - u \times(u \times (-\triangle)u)\big]\big)\big\|_{N}\\
&\lesssim \prod_{v = u,\tilde{u}} (1+\big\|v\big\|_{S})\big\|u\big\|_{S}\big(\sum_{k_1\in Z}2^{-\sigma|k-k_1|}\big\|P_{k_1}u\big\|_{S_{k_1}}\big).
\end{split}\end{equation}
We also have corresponding difference estimates: assuming that $u^{(j)}, j = 1,2$, map into $S^2$, while $\tilde{u}^{(j)}$ map into a small neighbourhood of $S^2$, then using the notation 
\[
\triangle_{1,2}F^{(j)} = F^{(1)} - F^{(2)}
\]
we have 
\begin{equation}\label{eq:multilin4}\begin{split}
&\big\|\triangle_{1,2}P_k\big[u^{(j)}(\nabla u^{(j)}\cdot\nabla u^{(j)} - \partial_t u^{(j)}\cdot\partial_t u^{(j)})\big]\big\|_{N}\\&\lesssim (1+\max_j\big\|u^{(j)}\big\|_{S})(\max_j\big\|u^{(j)}\big\|_{S})\big(\sum_{k_1\in Z}2^{-\sigma|k-k_1|}\big\|P_{k_1}u^{(1)} - P_ku^{(2)}\big\|_{S_{k_1}}\big)\\
& + (1+\max_j\big\|u^{(j)}\big\|_{S})(\big\|u^{(1)} - u^{(2)}\big\|_{S})\big(\max_j\sum_{k_1\in Z}2^{-\sigma|k-k_1|}\big\|P_{k_1}u^{(j)}\big\|_{S_{k_1}}\big)\\
\end{split}\end{equation}
and similarly 
\begin{equation}\label{eq:multilin5}\begin{split}
&\big\|P_k\triangle_{1,2}\big(\Pi_{\tilde{u}^{(j)}_{\perp}}\big((-\triangle)^{\frac{1}{2}}u^{(j)}\big)(u^{(j)}\cdot (-\triangle)^{\frac{1}{2}}u^{(j)})\big)\big\|_{N}\\&\lesssim \max_j\prod_{v = u^{(j)},\tilde{u}^{(j)}} (1+\big\|v\big\|_{S})\big\|u^{(j)}\big\|_{S}\big(\sum_{k_1\in Z}2^{-\sigma|k-k_1|}\big\|P_{k_1}u^{(1)} - P_{k_2}u^{(2)} \big\|_{S_{k_1}}\big)\\
& + \max_j\prod_{v = u^{(j)},\tilde{u}^{(j)}} (1+\big\|v\big\|_{S})\big\|u^{(1)}- u^{(2)}\big\|_{S}\big(\max_j\sum_{k_1\in Z}2^{-\sigma|k-k_1|}\big\|P_{k_1}u^{(j)}\big\|_{S_{k_1}}\big)\\
& +  \max_j(1+\big\|u^{(j)}\big\|_{S})\big\|\tilde{u}^{(1)}- \tilde{u}^{(2)}\big\|_{S}\big(\max_j\sum_{k_1\in Z}2^{-\sigma|k-k_1|}\big\|P_{k_1}u^{(j)}\big\|_{S_{k_1}}\big)\\
\end{split}\end{equation}
The analogous difference estimate for \eqref{eq:multilin3} is similar. 
In fact, in all these estimates the choice $\sigma = 1$ works. 
\end{prop}

\begin{proof} We shall only deal in detail with the case $n = 5$, since the case $n\geq 6$ is simpler due to the better decay with respect to large frequencies. Also, we note that then the desired estimates follow for a slightly different functional framework from the paper \cite{Ste}. We observe that the proof of \eqref{eq:multilin1} is really quite standard and follows for example from \cite{Tat3}. For completeness's sake, we include a simple version here. 
\\

{\it{Proof of \eqref{eq:multilin1}}}.  To achieve it, we localise the second and third factor to frequency $\sim 2^{k_1}, 2^{k_2}$, respectively, and we shall restrict the output logarithmic frequency $k$ to size $0$. This is possible on account of the scaling invariance of the estimate. We shall obtain exponential gains in terms of these frequencies in certain cases, and summation over all allowed frequencies will result in the desired bound \eqref{eq:multilin1}. 
\\

{\it{(1): high high interactions $\max\{k_1, k_2\}>10$.}} This is schematically written as 
\[
P_0[u\nabla_{t,x}u_{k_1}\nabla_{t,x}u_{k_2}]
\]
Then if $k_1 = k_2+O(1)$, we estimate this by 
\begin{align*}
&\big\|P_0[u\nabla_{t,x}u_{k_1}\nabla_{t,x}u_{k_2}]\big\|_{L_t^1 L_x^2}\lesssim \big\|\nabla_{t,x}u_{k_1}\big\|_{L_t^2 L_x^4}\big\|\nabla_{t,x}u_{k_2}\big\|_{L_t^2 L_x^4}\\
&\lesssim 2^{-\frac32 k_1}\prod_{j=1,2}\big\|u_{k_j}\big\|_{S_{k_j}}
\end{align*}
If $k_2>k_1+10$, say then we estimate it by 
\begin{align*}
&\big\|P_0[u\nabla_{t,x}u_{k_1}\nabla_{t,x}u_{k_2}]\big\|_{L_t^1 L_x^2}
= \big\|P_0[P_{k_2+O(1)}u\nabla_{t,x}u_{k_1}\nabla_{t,x}u_{k_2}]\big\|_{L_t^1 L_x^2}\\
&\lesssim \big\|P_{k_2+O(1)}u\big\|_{L_t^2L_x^4}\big\|\nabla_{t,x}u_{k_2}\big\|_{L_t^2 L_x^4}\big\|\nabla_{t,x}u_{k_1}\big\|_{L_t^\infty L_x^2+L_{t,x}^\infty}\\
&\lesssim 2^{-\frac{5}{2}k_2}\big\|u_{k_2}\big\|_{S_{k_2}}\big\|u_{k_2+O(1)}\big\|_{S_{k_2+O(1)}}\big\|u_{k_1}\big\|_{S_{k_1}}. 
\end{align*}
The case $k_1>k_2+10$ is of course the same. Summation over the suitable ranges of $k_1$, $k_2$ implies \eqref{eq:multilin1} in this case with $\sigma = \frac32$. 
\\

{\it{(2): high low interactions $\max\{k_1, k_2\}<-10$.}} Here one places $\nabla_{t,x}u_{k_j}$, $j = 1,2$, into $L_t^2 L_x^\infty$ and $u = P_{O(1)}u$ into $L_t^\infty L_x^2$. 
\\

{\it{(3): low high interactions $\max\{k_1, k_2\}\in [-10,10]$.}} This is the most delicate case. We may assume that $k_1<k_2-10$, since else we argue as in {\it{(1)}}. Thus $k_2\in [-10,10]$. Note that then
\begin{align*}
&\big\|P_0[u_{\geq k_1-10}\nabla_{t,x}u_{k_1}\nabla_{t,x}u_{k_2}]\big\|_{L_t^1 L_x^2}
\lesssim \big\|u_{\geq k_1-10}\big\|_{L_t^2 L_x^\infty}\big\|\nabla_{t,x}u_{k_1}\big\|_{L_t^2 L_x^\infty}\big\|\nabla_{t,x}u_{k_2}\big\|_{L_t^\infty L_x^2}\\
&\lesssim \big\|u\big\|_{S}\big\|u_{k_2}\big\|_{S_{k_2}}\big\|u_{k_1}\big\|_{S_{k_1}}, 
\end{align*}
which can be summed over $k_1<k_2-10$. 
One similarly estimates 
\[
P_0[Q_{\geq k_1-10}u_{<k_1-10}\nabla_{t,x}u_{k_1}\nabla_{t,x}u_{k_2}].
\]
We have now reduced to estimating 
\[
P_0[Q_{<k_1-10}u_{<k_1-10}\partial_{\alpha}u_{k_1}\partial^{\alpha}u_{k_2}]
\]
Here note that 
\begin{align*}
&\big\|P_0[Q_{<k_1-10}u_{<k_1-10}\partial_{\alpha}u_{k_1}\partial^{\alpha}Q_{>k_1-10}u_{k_2}]\big\|_{L_t^1 L_x^2}
\lesssim \big\|\partial_{\alpha}u_{k_1}\big\|_{L_t^2 L_x^\infty}\big\|\partial^{\alpha}Q_{>k_1-10}u_{k_2}\big\|_{L_{t,x}^2}\\
&\lesssim \big\|u_{k_1}\big\|_{S_{k_1}}\big\|u_{k_2}\big\|_{S_{k_2}},
\end{align*}
again summable over $k_1<k_2-10$. 
Also, we get
\begin{align*}
&\big\|P_0[Q_{<k_1-10}u_{<k_1-10}\partial_{\alpha}Q_{\geq k_1+10}u_{k_1}\partial^{\alpha}Q_{<k_1-10}u_{k_2}]\big\|_{X^{\frac{n}{2}-1,-\frac12,1}}\\
&\lesssim \sum_{j\geq k_1+10}2^{-\frac{j}{2}}\big\|Q_j\nabla_{t,x}u_{k_1}\big\|_{L_t^2 L_x^\infty}\big\|\partial^{\alpha}Q_{<k_1-10}u_{k_2}\big\|_{L_t^\infty L_x^2}\\
&\lesssim \big\|u_{k_1}\big\|_{S_{k_1}}\big\|u_{k_2}\big\|_{S_{k_2}}
\end{align*}
and hence summable over $k_1<k_2-10$.
Finally, we expand the expression out using its null-structure: 
\begin{align*}
&2P_0[Q_{<k_1-10}u_{<k_1-10}\partial_{\alpha}Q_{<k_1+10}u_{k_1}\partial^{\alpha}Q_{<k_1-10}u_{k_2}]\\& =  P_0[Q_{<k_1-10}u_{<k_1-10}\Box\big(Q_{<k_1+10}u_{k_1}Q_{<k_1-10}u_{k_2}\big)]\\
& - P_0[Q_{<k_1-10}u_{<k_1-10}\Box Q_{<k_1+10}u_{k_1}Q_{<k_1-10}u_{k_2}]\\
& - P_0[Q_{<k_1-10}u_{<k_1-10}Q_{<k_1+10}u_{k_1}\Box Q_{<k_1-10}u_{k_2}]\\
\end{align*}
Then we bound each of these: 
\begin{align*}
&P_0[Q_{<k_1-10}u_{<k_1-10}\Box\big(Q_{<k_1+10}u_{k_1}Q_{<k_1-10}u_{k_2}\big)]\\
& = \Box P_0[Q_{<k_1-10}u_{<k_1-10}\big(Q_{<k_1+10}u_{k_1}Q_{<k_1-10}u_{k_2}\big)]\\
& - P_0[\nabla_{t,x}Q_{<k_1-10}u_{<k_1-10}\nabla_{t,x}\big(Q_{<k_1+10}u_{k_1}Q_{<k_1-10}u_{k_2}\big)]\\
& - P_0[\nabla_{t,x}^2Q_{<k_1-10}u_{<k_1-10}\big(Q_{<k_1+10}u_{k_1}Q_{<k_1-10}u_{k_2}\big)]\\
\end{align*}
The last two terms on the right can be easily placed into $L_t^1 L_x^2$ using the $L_t^2 L_x^\infty$ norm for the low-frequency factors, while for the first term on the right, we get 
\begin{align*}
&\big\|\Box P_0[Q_{<k_1-10}u_{<k_1-10}\big(Q_{<k_1+10}u_{k_1}Q_{<k_1-10}u_{k_2}\big)]\big\|_{X^{\frac{n}{2}-1,-\frac12,1}}\\
&\lesssim \sum_{j<k_1+20}2^{\frac{j}{2}}\big\|P_0Q_j[Q_{<k_1-10}u_{<k_1-10}\big(Q_{<k_1+10}u_{k_1}Q_{<k_1-10}u_{k_2}\big)]\big\|_{L_{t,x}^2}\\
&\lesssim 2^{\frac{k_1}{2}}\big\|Q_{<k_1+10}u_{k_1}\big\|_{L_t^2 L_x^\infty}\big\|Q_{<k_1-10}u_{k_2}\big\|_{L_t^\infty L_x^2}\\
&\lesssim \big\|u_{k_1}\big\|_{S_{k_1}}\big\|u_{k_2}\big\|_{S_{k_2}}. 
\end{align*}
Further, we get 
\begin{align*}
&\big\|P_0[Q_{<k_1-10}u_{<k_1-10}\Box Q_{<k_1+10}u_{k_1}Q_{<k_1-10}u_{k_2}]\big\|_{L_t^1 L_x^2}\\&\lesssim \big\|\Box Q_{<k_1+10}u_{k_1}\big\|_{L_t^2 L_x^4}\big\|Q_{<k_1-10}u_{k_2}\big\|_{L_t^2 L_x^4}\\
&\lesssim 2^{\frac{k_1}{4}}\big\|u_{k_1}\big\|_{S_{k_1}}\big\|u_{k_2}\big\|_{S_{k_2}}. 
\end{align*}
To close things, we also get 
\begin{align*}
&\big\| P_0[Q_{<k_1-10}u_{<k_1-10}Q_{<k_1+10}u_{k_1}\Box Q_{<k_1-10}u_{k_2}]\big\|_{L_t^1 L_x^2}\\&
\lesssim \big\|Q_{<k_1+10}u_{k_1}\big\|_{L_t^2 L_x^\infty}\big\|\Box Q_{<k_1-10}u_{k_2}\big\|_{L_{t,x}^2}
\lesssim \big\|u_{k_1}\big\|_{S_{k_1}}\big\|u_{k_2}\big\|_{S_{k_2}}
\end{align*}
and the desired bound follows again by summing over $k_1<k_2-10$. 
\\

{\it{Proof of \eqref{eq:multilin2}}}. Here we shall be able to get by only using Strichartz type norms, by taking advantage of the condition $u\cdot u = 1$. Using the standard Littlewood-Paley trichotomy, we have that 
\begin{equation}\label{eq:orthomicro}
0 = u\cdot u - p\cdot p = \sum_{|k_1-k_2|\leq 10} u_{k_1}u_{k_2} + 2\sum_{k_1}u_{k_1}\cdot u_{<k_1-10}
\end{equation}
This implies that 
\begin{equation}\label{eq:cancel}
0 =  \sum_{|k_1-k_2|<10} (-\triangle)^{\frac{1}{2}}\big(u_{k_1}u_{k_2}\big) + 2\sum_{k_1}(-\triangle)^{\frac{1}{2}}\big(u_{k_1}\cdot u_{<k_1-10}\big)
\end{equation}
Here the first term is better, since the outer derivative falls on the low-frequency output. We shall use this to replace the second term on the right by the first.  Write 
\begin{equation}\label{eq:miltilin2decomp}\begin{split}
\Pi_{\tilde{u}_{\perp}}\big((-\triangle)^{\frac{1}{2}}u\big)(u\cdot (-\triangle)^{\frac{1}{2}}u) &= \sum_{|k_1-k_2|\leq 10}\Pi_{\tilde{u}_{\perp}}\big((-\triangle)^{\frac{1}{2}}u\big)(u_{k_1}\cdot (-\triangle)^{\frac{1}{2}}u_{k_2})\\
&+ \sum_{k_1}\Pi_{\tilde{u}_{\perp}}\big((-\triangle)^{\frac{1}{2}}u\big)(u_{k_1}\cdot (-\triangle)^{\frac{1}{2}}u_{<k_1-10})\\
&+\sum_{k_2}\Pi_{\tilde{u}_{\perp}}\big((-\triangle)^{\frac{1}{2}}u\big)(u_{<k_2-10}\cdot (-\triangle)^{\frac{1}{2}}u_{k_2})\\
\end{split}\end{equation}
Then for the first term on the right we infer 
\begin{align*}
&\big\|P_0\big[ \sum_{|k_1-k_2|\leq 10}\Pi_{\tilde{u}_{\perp}}\big((-\triangle)^{\frac{1}{2}}u\big)(u_{k_1}\cdot (-\triangle)^{\frac{1}{2}}u_{k_2})\big]\big\|_{L_t^1 L_x^2}\\
&\lesssim \sum_{\substack{|k_1-k_2|\leq 10\\ k_1<-20}}\big\|P_{[-20,20]}\big(\Pi_{\tilde{u}_{\perp}}\big((-\triangle)^{\frac{1}{2}}u\big)\big)\big\|_{L_t^\infty L_x^2}\big\|u_{k_1}\big\|_{L_t^2 L_x^\infty}\big\|(-\triangle)^{\frac{1}{2}}u_{k_2}\big\|_{L_t^2 L_x^\infty}\\
& +  \sum_{\substack{|k_1-k_2|\leq 10\\ k_1\geq -20}}\big\|\Pi_{\tilde{u}_{\perp}}\big((-\triangle)^{\frac{1}{2}}u\big)\big\|_{L_t^\infty L_x^2+L_{t,x}^\infty}\big\|u_{k_1}\big\|_{L_t^2 L_x^4}\big\|(-\triangle)^{\frac{1}{2}}u_{k_2}\big\|_{L_t^2 L_x^4}\\
\end{align*}
Then using a further elementary frequency decomposition it is easy to see that
\begin{align*}
&\big\|P_{[-20,20]}\big(\Pi_{\tilde{u}_{\perp}}\big((-\triangle)^{\frac{1}{2}}u\big)\big)\big\|_{L_t^\infty L_x^2}\lesssim \sum _{k_3\in Z}2^{-|k_3|}\big\|P_{k_3}u\big\|_{S_{k_1}}(1+\big\|\tilde{u}\big\|_{S}),\\
&\big\|\big(\Pi_{\tilde{u}_{\perp}}\big((-\triangle)^{\frac{1}{2}}u\big)\big)\big\|_{L_t^\infty L_x^2+L_{t,x}^{\infty}}\lesssim \sum _{k_3\in Z}2^{-|k_3|}\big\|P_{k_3}u\big\|_{S_{k_1}}(1+\big\|\tilde{u}\big\|_{S})
\end{align*}
and so we obtain that 
\begin{align*}
&\sum_{\substack{|k_1-k_2|\leq 10\\ k_1<-20}}\big\|P_{[-20,20]}\big(\Pi_{\tilde{u}_{\perp}}\big((-\triangle)^{\frac{1}{2}}u\big)\big)\big\|_{L_t^\infty L_x^2}\big\|u_{k_1}\big\|_{L_t^2 L_x^\infty}\big\|(-\triangle)^{\frac{1}{2}}u_{k_2}\big\|_{L_t^2 L_x^\infty}\\
&\lesssim \sum_{\substack{|k_1-k_2|\leq 10\\ k_1<-20}}2^{\frac{k_2-k_1}{2}}\prod_{j=1,2}\big\|P_{k_j}u\big\|_{S_{k_j}}\big(\sum _{k_3\in Z}2^{-|k_3|}\big\|P_{k_3}u\big\|_{S_{k_1}}(1+\big\|\tilde{u}\big\|_{S})\big)\\
&\lesssim \big(\sum _{k_3\in Z}2^{-|k_3|}\big\|P_{k_3}u\big\|_{S_{k_1}}\big)\big\|u\big\|_{S}^2(1+\big\|\tilde{u}\big\|_{S}),
\end{align*}
as well as 
\begin{align*}
&\sum_{\substack{|k_1-k_2|\leq 10\\ k_1\geq -20}}\big\|\Pi_{\tilde{u}_{\perp}}\big((-\triangle)^{\frac{1}{2}}u\big)\big\|_{L_t^\infty L_x^2+L_{t,x}^\infty}\big\|u_{k_1}\big\|_{L_t^2 L_x^4}\big\|(-\triangle)^{\frac{1}{2}}u_{k_2}\big\|_{L_t^2 L_x^4}\\
&\lesssim \big(\sum_{\substack{|k_1-k_2|\leq 10\\ k_1\geq -20}}2^{-\frac{5}{2}k_1}\big\|u_{k_1}\big\|_{S_{k_1}}\big\|u_{k_2}\big\|_{S_{k_2}}\big)\sum _{k_3\in Z}2^{-|k_3|}\big\|P_{k_3}u\big\|_{S_{k_1}}(1+\big\|\tilde{u}\big\|_{S})\\
&\lesssim \big(\sum _{k_3\in Z}2^{-|k_3|}\big\|P_{k_3}u\big\|_{S_{k_1}}\big)\big\|u\big\|_{S}^2(1+\big\|\tilde{u}\big\|_{S}).
\end{align*}
This concludes the required bound for the first term on the right hand side of \eqref{eq:miltilin2decomp}. 
\\

Now we pass to the second term. We write it as a sum of three terms:
\begin{align*}
&\sum_{k_1}\Pi_{\tilde{u}_{\perp}}\big((-\triangle)^{\frac{1}{2}}u\big)(u_{k_1}\cdot (-\triangle)^{\frac{1}{2}}u_{<k_1-10})\\&= \sum_{k_1\geq 5}\Pi_{\tilde{u}_{\perp}}\big((-\triangle)^{\frac{1}{2}}u\big)(u_{k_1}\cdot (-\triangle)^{\frac{1}{2}}u_{<k_1-10})\\
& +  \sum_{k_1\in[-5,5]}\Pi_{\tilde{u}_{\perp}}\big((-\triangle)^{\frac{1}{2}}u\big)(u_{k_1}\cdot (-\triangle)^{\frac{1}{2}}u_{<k_1-10})\\
& +  \sum_{k_1<-5}\Pi_{\tilde{u}_{\perp}}\big((-\triangle)^{\frac{1}{2}}u\big)(u_{k_1}\cdot (-\triangle)^{\frac{1}{2}}u_{<k_1-10})\\
\end{align*}
Then we get 
\begin{align*}
&\big\|P_0\big( \sum_{k_1\geq 5}\Pi_{\tilde{u}_{\perp}}\big((-\triangle)^{\frac{1}{2}}u\big)(u_{k_1}\cdot (-\triangle)^{\frac{1}{2}}u_{<k_1-10})\big)\big\|_{L_t^1 L_x^2}\\
&\lesssim  \sum_{k_1\geq 5}\big\|P_{[k_1 - 5,k_1+5]}\big(\Pi_{\tilde{u}_{\perp}}\big((-\triangle)^{\frac{1}{2}}u\big)\big)\big\|_{L_t^\infty L_x^2}\big\|u_{k_1}\big\|_{L_t^2 L_x^\infty}\big\|(-\triangle)^{\frac{1}{2}}u_{<k_1-10}\big\|_{L_t^2 L_x^\infty}\\
&\lesssim  \sum_{k_1\geq 5, k_2<k_1-10}\sum_{k_3}2^{-\frac{3}{2}|k_3|}\big\|P_{k_3}u\big\|_{S_{k_3}}(1+\big\|\tilde{u}\big\|_{S})\big\|u_{k_1}\big\|_{S_{k_1}}\big\|u_{k_2}\big\|_{S_{k_2}}\\
&\lesssim \big(\sum_{k_3}2^{-\frac{3}{2}|k_3|}\big\|P_{k_3}u\big\|_{S_{k_3}}\big)(1+\big\|\tilde{u}\big\|_{S})\big\|u\big\|_{S}^2. 
\end{align*}
Similarly, for the term of intermediate $k_1$, we have 
\begin{align*}
&\big\|P_0\big[ \sum_{k_1\in[-5,5]}\Pi_{\tilde{u}_{\perp}}\big((-\triangle)^{\frac{1}{2}}u\big)(u_{k_1}\cdot (-\triangle)^{\frac{1}{2}}u_{<k_1-10})\big]\big\|_{L_t^1 L_x^2}\\
&\lesssim \sum_{k_1\in[-5,5]}\big\|P_{<10}\Pi_{\tilde{u}_{\perp}}\big((-\triangle)^{\frac{1}{2}}u\big)\big\|_{L_t^2 L_x^\infty}\big\|u_{k_1}\big\|_{S_{k_1}}\big\|(-\triangle)^{\frac{1}{2}}u_{<k_1-10}\big\|_{L_t^2 L_x^\infty},
\end{align*}
and one closes by observing that 
\[
\big\|P_{<10}\Pi_{\tilde{u}_{\perp}}\big((-\triangle)^{\frac{1}{2}}u\big)\big\|_{L_t^2 L_x^\infty}\lesssim (\big\|\tilde{u}\big\|_{S}+1)\big\|u\big\|_{S},\,\big\|(-\triangle)^{\frac{1}{2}}u_{<k_1-10}\big\|_{L_t^2 L_x^\infty}\lesssim \big\|u\big\|_{S}. 
\]
Finally, for the range of low $k_1<-5$, we place both $u_{k_1},  (-\triangle)^{\frac{1}{2}}u_{<k_1-10}$ into $L_t^2L_x^\infty$ and observe that 
\[
\big\|u_{k_1}\big\|_{L_t^2L_x^\infty}\big\| (-\triangle)^{\frac{1}{2}}u_{<k_1-10}\big\|_{L_t^2 L_x^\infty}\lesssim \big\|u_{k_1}\big\|_{S_{k_1}}\big\|u\big\|_{S}. 
\]
Then we close by using that 
\begin{align*}
&P_0\big(\sum_{k_1<-5}\Pi_{\tilde{u}_{\perp}}\big((-\triangle)^{\frac{1}{2}}u\big)(u_{k_1}\cdot (-\triangle)^{\frac{1}{2}}u_{<k_1-10})\big)\\
& = P_0\big(\sum_{k_1<-5}P_{[-2,2]}\big(\Pi_{\tilde{u}_{\perp}}\big((-\triangle)^{\frac{1}{2}}u\big)\big)(u_{k_1}\cdot (-\triangle)^{\frac{1}{2}}u_{<k_1-10})\big),\\
\end{align*}
as well as 
\[
\big\|P_{[-2,2]}\big(\Pi_{\tilde{u}_{\perp}}\big((-\triangle)^{\frac{1}{2}}u\big)\big)\big\|_{L_t^\infty L_x^2}\lesssim (1+\big\|\tilde{u}\big\|_{S})\big(\sum_{k_3}2^{-\frac{3}{2}|k_3|}\big\|P_{k_3}u\big\|_{S_{k_3}}\big). 
\]
This concludes the required bound for the second term on the right in \eqref{eq:miltilin2decomp}. 
\\

Finally, the third term in  \eqref{eq:miltilin2decomp} is the most delicate, as the derivative $(-\triangle)^{\frac12}$ lands on the higher -frequency term $u_{k_2}$. To deal with it, we note, using Lemma~\ref{lem:multilinestimates}, that the difference 
\begin{align*}
\sum_{k_2}\Pi_{\tilde{u}_{\perp}}\big((-\triangle)^{\frac{1}{2}}u\big)(u_{<k_2-10}\cdot (-\triangle)^{\frac{1}{2}}u_{k_2}) - \sum_{k_2}\Pi_{\tilde{u}_{\perp}}\big((-\triangle)^{\frac{1}{2}}u\big)(-\triangle)^{\frac{1}{2}}(u_{<k_2-10}\cdot u_{k_2})
\end{align*}
can be estimated like the second term on the right in \eqref{eq:miltilin2decomp}, and hence it suffices to bound 
\begin{align*}
&\sum_{k_2}\Pi_{\tilde{u}_{\perp}}\big((-\triangle)^{\frac{1}{2}}u\big)(-\triangle)^{\frac{1}{2}}(u_{<k_2-10}\cdot u_{k_2})\\& = -\sum_{|k_3 - k_4|<10}\frac12\Pi_{\tilde{u}_{\perp}}\big((-\triangle)^{\frac{1}{2}}u\big)(-\triangle)^{\frac{1}{2}}(u_{k_3}\cdot u_{k_4})
\end{align*}
where we have used \eqref{eq:orthomicro}. This term is again straightforward to estimate: we have 
\begin{align*}
&\big\|P_0\big[\sum_{\substack{|k_3 - k_4|<10\\ k_3<-20}}\frac12\Pi_{\tilde{u}_{\perp}}\big((-\triangle)^{\frac{1}{2}}u\big)(-\triangle)^{\frac{1}{2}}(u_{k_3}\cdot u_{k_4})\big]\big\|_{L_t^1 L_x^2}\\
&\lesssim \sum_{\substack{|k_3 - k_4|<10\\ k_3<-20}}\big\|P_{[-10,10]}\big[\Pi_{\tilde{u}_{\perp}}\big((-\triangle)^{\frac{1}{2}}u\big)\big]\big\|_{L_t^\infty L_x^2}\big\|(-\triangle)^{\frac{1}{2}}(u_{k_3}\cdot u_{k_4})\big\|_{L_t^1 L_x^\infty}
\end{align*}
and we close for the case $k_3<-20$ by observing that 
\begin{align*}
\sum_{\substack{|k_3 - k_4|<10\\ k_3<-20}}\big\|(-\triangle)^{\frac{1}{2}}(u_{k_3}\cdot u_{k_4})\big\|_{L_t^1 L_x^\infty}&\lesssim \sum_{\substack{|k_3 - k_4|<10\\ k_3<-20}}2^{k_3}\big\|u_{k_3}\big\|_{L_t^2L_x^\infty}\big\|u_{k_4}\big\|_{L_t^2L_x^\infty}\\
&\lesssim \big\|u\big\|_{S}^2, 
\end{align*}
as well as 
\[
\big\|P_{[-10,10]}\big[\Pi_{\tilde{u}_{\perp}}\big((-\triangle)^{\frac{1}{2}}u\big)\big]\big\|_{L_t^\infty L_x^2}\lesssim (1+\big\|\tilde{u}\big\|_{S})\sum_{k_3}2^{-\frac{3}{2}|k_3|}\big\|P_{k_3}u\big\|_{S_{k_3}}. 
\]
On the other hand, if $k_3>-20$, we place both $u_{k_{3,4}}$ into $L_t^2 L_x^4$. We omit the simple details. This finally concludes the bound of estimate \eqref{eq:multilin2}. 
\\

{\it{Proof of \eqref{eq:multilin3}}}. We commence by observing that we may in fact get rid of the outer operator $\Pi_{\tilde{u}^{\perp}}$, since one easily checks that 
\[
\big\|P_0\big[\Pi_{\tilde{u}^{\perp}}F\big]\big\|_{L_t^1 L_x^2}\lesssim (1+\big\|\tilde{u}\big\|_{S})\sum_{k_1}2^{-|k_1|}\big\|P_{k_1}F\big\|_{L_t^1\dot{H}^{\frac{n}{2}-1}}.
\]
Then assuming that we have proved the bound 
\[
\big\|P_{k_1}F\big\|_{L_t^1\dot{H}^{\frac{n}{2}-1}}\lesssim \sum_{k_2}2^{-\sigma|k_1-k_2|}\big\|P_{k_2}u\big\|_{S_{k_2}}
\]
for some $\sigma>1$, we then infer the bound 
\[
\big\|P_0\big[\Pi_{\tilde{u}^{\perp}}F\big]\big\|_{L_t^1 L_x^2}\lesssim (1+\big\|\tilde{u}\big\|_{S})\sum_{k_2}2^{-|k_2|}\big\|P_{k_2}u\big\|_{S_{k_2}}
\]
Next, localising the last two factors to dyadic frequencies, and the output to frequency $\sim 1$ as we may, we arrive at the expression 
\[
P_0\big[u \times (-\triangle)^{\frac{1}{2}}(u_{k_1} \times (-\triangle)^{\frac{1}{2}}u_{k_2}) - u \times(u_{k_1} \times (-\triangle)u_{k_2})\big]
\]
Then we first dispose of the easy cases: 
\\

{\it{Both frequencies large: $\max\{k_1,k_2\}>10$.}} If $k_1 = k_2+O(1)$, we simply place both high frequency factors into $L_t^2 L_x^4$, resulting in 
\begin{align*}
&\big\|P_0\big[u \times (-\triangle)^{\frac{1}{2}}(u_{k_1} \times (-\triangle)^{\frac{1}{2}}u_{k_2}) - u \times(u_{k_1} \times (-\triangle)u_{k_2})\big]\big\|_{L_t^1 L_x^2}\\
&\lesssim 2^{2{k_1}}\big\|P_{k_1}u\big\|_{L_t^2 L_x^4}\big\|u_{k_2}\big\|_{L_t^2 L_x^4}\lesssim 2^{2k_1 - (1+\frac{5}{2})k_1}\prod_{j=1,2}\big\|u_{k_j}\big\|_{S_{k_j}}, 
\end{align*}
whence we have 
\begin{align*}
&\big\|\sum_{k_1 = k_2+O(1)>10}P_0\big[u \times (-\triangle)^{\frac{1}{2}}(u_{k_1} \times (-\triangle)^{\frac{1}{2}}u_{k_2}) - u \times(u_{k_1} \times (-\triangle)u_{k_2})\big]\big\|_{L_t^1 L_x^2}\\
&\lesssim \sum_{k_1 = k_2+O(1)>10}2^{-\frac {3}{2}|k_1|}\prod_{j=1,2}\big\|u_{k_j}\big\|_{S_{k_j}}\lesssim \big(\sum_{k_1}2^{-\frac {3}{2}|k_1|}\big\|P_{k_1}u\big\|_{S_{k_1}}\big)\big\|u\big\|_{S}. 
\end{align*}
On the other hand, if $k_2\gg k_1$, we use 
\begin{align*}
&P_0\big[u \times (-\triangle)^{\frac{1}{2}}(u_{k_1} \times (-\triangle)^{\frac{1}{2}}u_{k_2}) - u \times(u_{k_1} \times (-\triangle)u_{k_2})\big]\\&
= P_0\big[P_{k_2+O(1)}u \times (-\triangle)^{\frac{1}{2}}(u_{k_1} \times (-\triangle)^{\frac{1}{2}}u_{k_2}) - P_{k_2+O(1)}u \times(u_{k_1} \times (-\triangle)u_{k_2})\big].
\end{align*}
Then place the first and third factor into $L_t^2 L_x^4$ and the middle factor into $L_t^\infty L_x^2 + L_{t,x}^{\infty}$. The case $k_2\ll k_1$ is similar. 
\\

{\it{Both frequencies small: $\max\{k_1,k_2\}<-10$.}} Here we observe that Lemma~\ref{lem:multilinestimates} allows us to place one derivative $(-\triangle)^{\frac12}$ onto the factor $u_{k_1}$, even if $k_1<k_2-10$. Thus we reduce to bounding the schematic expression 
\[
P_0\big[P_{[-5,5]}u \nabla_x u_{k_1}\nabla_x u_{k_2}\big], 
\]
which is straightforward since we can place the second and third factor into $L_t^2 L_x^\infty$. We omit the simple details. 
\\

{\it{One frequency intermediate, the other small}}:  $\max\{k_1,k_2\}\in [-10,10]$. This case is a bit more difficult, and we shall exploit the geometric structure of the expression. 
We split this further into two cases:
\\

{\it{(i): $k_1\in [-10,10], k_2<10$.}} Here the difference structure inherent in the term is not helpful. In fact, we can immediately estimate 
\begin{align*}
\big\|P_0\big[ u \times(u_{k_1} \times (-\triangle)u_{k_2})\big]\big\|_{L_t^1 L_x^2}
&\lesssim \big\|u_{k_1} \big\|_{L_t^2 L_x^4}\big\| (-\triangle)u_{k_2})\big\|_{L_t^2 L_x^4}\\
&\lesssim 2^{\frac{k_2}{4}}\big\|u_{k_2}\big\|_{S_{k_2}}\big\|u_{k_1}\big\|_{S_{k_1}}, 
\end{align*}
and here of course we can sum over $k_2<10$ to infer the desired bound. Next, using Lemma~\ref{lem:multilinestimates} allows us to replace the term 
\[
P_0\big[ u \times (-\triangle)^{\frac12}(u_{k_1} \times (-\triangle)^{\frac12}u_{k_2})\big]
\]
by 
\[
P_0\big[ u \times ((-\triangle)^{\frac12}u_{k_1} \times (-\triangle)^{\frac12}u_{k_2})\big]
\]
up to a term which is estimated like $P_0\big[ u \times(u_{k_1} \times (-\triangle)u_{k_2})\big]$. Before exploiting the algebraic structure of the term above, we reduce the first factor $u$ to frequency $<2^{k_2-10}$, which we can on account of 
\begin{align*}
&\big\|P_0\big[ u_{\geq k_2-10} \times ((-\triangle)^{\frac12}u_{k_1} \times (-\triangle)^{\frac12}u_{k_2})\big]\big\|_{L_t^1 L_x^2}\\&\lesssim \big\|u_{\geq k_2-10}\big\|_{L_t^2L_x^\infty}\big\|(-\triangle)^{\frac12}u_{k_1}\big\|_{L_t^\infty L_x^2}\big\| (-\triangle)^{\frac12}u_{k_2}\big\|_{L_t^2 L_x^\infty}\\
&\lesssim \big\|u_{k_1}\big\|_{S_{k_1}}\big\|u_{k_2}\big\|_{S_{k_2}}\big\|u\big\|_{S}. 
\end{align*}
 Summing over $k_2<10$ and recalling that $k_1\in [-10,10]$ leads to the desired bound. 
 \\
 Consider now the expression 
\[
P_0\big[ u_{<k_2-10} \times ((-\triangle)^{\frac12}u_{k_1} \times (-\triangle)^{\frac12}u_{k_2})\big]
\] 
Write this as 
\begin{align*}
&P_0\big[ u_{<k_2-10} \times ((-\triangle)^{\frac12}u_{k_1} \times (-\triangle)^{\frac12}u_{k_2})\big]\\
& = P_0\big[(-\triangle)^{\frac12}u_{k_1}( u_{<k_2-10} \cdot (-\triangle)^{\frac12}u_{k_2}) - (-\triangle)^{\frac12}u_{k_2}(u_{<k_2-10} \cdot(-\triangle)^{\frac12}u_{k_1})\big]
\end{align*}
In order to estimate this, we use a frequency localised version of \eqref{eq:orthomicro}. Specifically, we have 
\begin{equation}\label{eq:orthomicrolocalized}
0 = 2u_k\cdot u_{<k-10} + \sum_{k_1 = k_2+O(1)}P_k(u_{k_1}\cdot u_{k_2}) + 2^{-k}L(u_k, \nabla_x u_{<k-10}), 
\end{equation}
where $L$ is a bilinear operator of the form used in Lemma~\ref{lem:multilinestimates} with a bounded kernel $m(\xi, \eta)$. We conclude the schematic relation 
\begin{align*}
(-\triangle)^{\frac12}u_k\cdot u_{<k-10} =& -\frac12 (-\triangle)^{\frac12} \sum_{k_1 = k_2+O(1)}P_k(u_{k_1}\cdot u_{k_2})\\
& + L(u_k, \nabla_x u_{<k-10})
\end{align*}
It follows that we can write 
\begin{align*}
&P_0\big[(-\triangle)^{\frac12}u_{k_1}( u_{<k_2-10} \cdot (-\triangle)^{\frac12}u_{k_2})\\
& = -\frac12 P_0\big[(-\triangle)^{\frac12}u_{k_1}\sum_{k_3 = k_4+O(1)}(-\triangle)^{\frac12} P_{k_2}(u_{k_3}\cdot u_{k_4})\big]\\
& + P_0\big[(-\triangle)^{\frac12}u_{k_1}L(u_{k_2}, \nabla_x u_{<k_2-10})\big],
\end{align*}
and here we have (keeping in mind that $k_1\in [-10,10]$)
\begin{align*}
&\big\|P_0\big[(-\triangle)^{\frac12}u_{k_1}\sum_{k_3 = k_4+O(1)}(-\triangle)^{\frac12} P_{k_2}(u_{k_3}\cdot u_{k_4})\big]\big\|_{L_t^1 L_x^2}\\&
\lesssim 2^{k_2}\sum_{k_3 = k_4+O(1)\geq k_2}\big\|(-\triangle)^{\frac12}u_{k_1}\big\|_{L_t^\infty L_x^2}\big\|u_{k_3}\big\|_{L_t^2 L_x^\infty}\big\|u_{k_4}\big\|_{L_t^2 L_x^\infty}\\
&\lesssim \big\|(-\triangle)^{\frac12}u_{k_1}\big\|_{L_t^\infty L_x^2}\sum_{k_3 = k_4+O(1)\geq k_2}2^{k_2 - k_3}\big\|u_{k_3}\big\|_{S_{k_3}}\big\|u_{k_4}\big\|_{S_{k_4}}, 
\end{align*}
and here we can sum over $k_2<10$ to arrive at an upper bound of $\lesssim \big\|u_{k_1}\big\|_{S_{k_1}}\big\|u\big\|_{S}^2$, as desired. 
We also have the simple bound 
\begin{align*}
&\big\|P_0\big[(-\triangle)^{\frac12}u_{k_1}L(u_{k_2}, \nabla_x u_{<k_2-10})\big]\big\|_{L_t^1 L_x^2}\\&
\lesssim \big\|(-\triangle)^{\frac12}u_{k_1}\big\|_{L_t^\infty L_x^2}\big\|u_{k_2}\big\|_{L_t^2 L_x^\infty}\big\|\nabla_x u_{<k_2-10}\big\|_{L_t^2 L_x^\infty}\\
&\lesssim \big\|(-\triangle)^{\frac12}u_{k_1}\big\|_{L_t^\infty L_x^2}\big\|u_{k_2}\big\|_{S_{k_2}}\big\|u\big\|_{S}, 
\end{align*}
and summing over $k_2<10$, we arrive again at the bound $\lesssim \big\|u_{k_1}\big\|_{S_{k_1}}\big\|u\big\|_{S}^2$. This concludes the case {{(i)}}. 
\\

{\it{(ii)}}: $k_2\in [-10,10], k_1<10$. Proceeding in analogy to case {\it{(i)}}, we immediately reduce to the expression 
\[
P_0\big[u_{<k_1-10} \times (-\triangle)^{\frac{1}{2}}(u_{k_1} \times (-\triangle)^{\frac{1}{2}}u_{k_2}) - u_{<k_1-10} \times(u_{k_1} \times (-\triangle)u_{k_2})\big]
\]
Here we first note that on account of Lemma~\ref{lem:multilinestimates} we have 
\begin{align*}
&\big\|P_0\big[u_{<k_1-10} \times (-\triangle)^{\frac{1}{2}}(u_{k_1} \times (-\triangle)^{\frac{1}{2}}u_{k_2}) - (-\triangle)^{\frac{1}{2}}\big(u_{<k_1-10} \times (u_{k_1} \times (-\triangle)^{\frac{1}{2}}u_{k_2})\big)\big]\big\|_{L_t^1 L_x^2}\\
& \lesssim \big\|(-\triangle)^{\frac{1}{2}}u_{<k_1-10}\big\|_{L_t^2 L_x^\infty}\big\|u_{k_1}\big\|_{L_t^2 L_x^\infty}\big\|(-\triangle)^{\frac{1}{2}}u_{k_2}\big\|_{L_t^\infty L_x^2}\lesssim \big\|u\big\|_{S}\big\|u_{k_1}\big\|_{S_{k_1}}\big\|u_{k_2}\big\|_{S_{k_2}}, 
\end{align*}
Then summation over $k_1<10$ gives the required bound. 
\\

Next, we expand out 
\begin{equation}\label{eq:thefourterms}\begin{split}
&P_0\big[(-\triangle)^{\frac{1}{2}}\big(u_{<k_1-10} \times (u_{k_1} \times (-\triangle)^{\frac{1}{2}}u_{k_2})\big) - u_{<k_1-10} \times(u_{k_1} \times (-\triangle)u_{k_2})\big]\\
& = P_0(-\triangle)^{\frac{1}{2}}\big(u_{k_1}(u_{<k_1-10}\cdot (-\triangle)^{\frac{1}{2}}u_{k_2}) -  (-\triangle)^{\frac{1}{2}}u_{k_2}(u_{<k_1-10}\cdot u_{k_1})\big)\\
& - P_0\big(u_{k_1}( u_{<k_1-10}\cdot (-\triangle)u_{k_2}) - (-\triangle)u_{k_2}(u_{<k_1-10} \cdot u_{k_1})\big).
\end{split}\end{equation}
Then pairing up these last 4 terms suitably, we have 
\begin{align*}
&P_0(-\triangle)^{\frac{1}{2}}\big(u_{k_1}(u_{<k_1-10}\cdot (-\triangle)^{\frac{1}{2}}u_{k_2})\big) -  P_0\big(u_{k_1}( u_{<k_1-10}\cdot (-\triangle)u_{k_2})\big)\\
& = P_0(-\triangle)^{\frac{1}{2}}\big(u_{k_1}(-\triangle)^{\frac{1}{2}}(u_{<k_1-10}\cdot u_{k_2}) - P_0\big(u_{k_1}(-\triangle)( u_{<k_1-10}\cdot u_{k_2})\big)\\
& + u_{k_1}L(\nabla_x u_{<k_1-10}, u_{k_2})\\
& = L((-\triangle)^{\frac{1}{2}}u_{k_1}, (-\triangle)^{\frac{1}{2}}(u_{<k_1-10}\cdot u_{k_2})) +  u_{k_1}L(\nabla_x u_{<k_1-10}, u_{k_2}).\\
\end{align*}
The last term is straightforward since 
\begin{align*}
\big\| u_{k_1}L(\nabla_x u_{<k_1-10}, u_{k_2})\big\|_{L_t^1 L_x^2}&\lesssim \big\|u_{k_1}\big\|_{L_t^2 L_x^\infty}\big\|\nabla_x u_{<k_1-10}\big\|_{L_t^2 L_x^\infty}\big\|u_{k_2}\big\|_{L_t^\infty L_x^2}\\
&\lesssim \big\|u\big\|_{S}\big\|u_{k_1}\big\|_{S_{k_1}}\big\|u_{k_2}\big\|_{S_{k_2}}, 
\end{align*}
and we can sum over $k_1<10$. Further, we see that 
\begin{align*}
 L((-\triangle)^{\frac{1}{2}}u_{k_1}, (-\triangle)^{\frac{1}{2}}(u_{<k_1-10}\cdot u_{k_2})) =  L((-\triangle)^{\frac{1}{2}}u_{k_1}, (-\triangle)^{\frac{1}{2}}(u_{<k_2-10}\cdot u_{k_2})) + \text{error}, 
\end{align*}
where the term $\text{error}$ here is estimated exactly like the previous term.  But then taking advantage of \eqref{eq:orthomicrolocalized}, we find 
\begin{align*}
&L((-\triangle)^{\frac{1}{2}}u_{k_1}, (-\triangle)^{\frac{1}{2}}(u_{<k_2-10}\cdot u_{k_2}))\\
& = -\frac12 \sum_{k_3 = k_4+O(1)>k_2}L((-\triangle)^{\frac{1}{2}}u_{k_1}, (-\triangle)^{\frac{1}{2}}P_{k_2}(u_{k_3}\cdot u_{k_4}))\\
& + 2^{-k_2}L((-\triangle)^{\frac{1}{2}}u_{k_1}, (-\triangle)^{\frac{1}{2}}L(\nabla_x u_{<k_2-10}, u_{k_2})).\\
\end{align*}
Then we have 
\begin{align*}
&\big\| -\frac12 \sum_{k_3 = k_4+O(1)>k_2}L((-\triangle)^{\frac{1}{2}}u_{k_1}, (-\triangle)^{\frac{1}{2}}P_{k_2}(u_{k_3}\cdot u_{k_4}))\big\|_{L_t^1 L_x^2}\\
&\lesssim \sum_{k_3 = k_4+O(1)>k_2}2^{k_2}\big\|(-\triangle)^{\frac{1}{2}}u_{k_1}\big\|_{L_t^2 L_x^\infty}\big\|u_{k_3}\big\|_{L_t^2 L_x^\infty}\big\|u_{k_4}\big\|_{L_t^\infty L_x^2}
\end{align*}
The preceding sum can be further bounded by 
\begin{align*}
&\lesssim \sum_{k_3 = k_4+O(1)>k_2}2^{k_2}2^{\frac{k_1 - k_3}{2}}2^{-\frac{5}{2}k_4}\big\|u_{k_1}\big\|_{S_{k_1}}\big\|u_{k_3}\big\|_{S_{k_3}}\big\|u_{k_4}\big\|_{S_{k_4}}\\&\lesssim \big(\sum_{k_1}2^{-|k_4 - k_2|}\big\|u_{k_4}\big\|_{S_{k_4}}\big\|u\big\|_{S}\big)\big\|u_{k_1}\big\|_{S_{k_1}}. 
\end{align*}
This can be summed over $k_1<10$ to yield the desired kind of bound. 
\\

Finally, we have the simpler bound 
\begin{align*}
&\big\|2^{-k_2}L((-\triangle)^{\frac{1}{2}}u_{k_1}, (-\triangle)^{\frac{1}{2}}L(\nabla_x u_{<k_2-10}, u_{k_2}))\big\|_{L_t^1 L_x^2}\\&\lesssim 
\big\|(-\triangle)^{\frac{1}{2}}u_{k_1}\big\|_{L_t^2 L_x^\infty}\big\|\nabla_x u_{<k_2-10}\big\|_{L_t^2 L_x^\infty}\big\|u_{k_2}\big\|_{L_t^\infty L_x^2},
\end{align*}
which after summation over $k_1<10$ is again bounded by $\lesssim \big\|u\big\|_{S}^2\big\|u_{k_2}\big\|_{S_{k_2}}$. 
\\

Returning to \eqref{eq:thefourterms}, it remains to bound the difference 
\begin{align*}
&P_0(-\triangle)^{\frac{1}{2}}\big[(-\triangle)^{\frac{1}{2}}u_{k_2}(u_{<k_1-10}\cdot u_{k_1})\big] - P_0\big[(-\triangle)u_{k_2}(u_{<k_1-10} \cdot u_{k_1})\big]\\
& = -\frac12 L((-\triangle)^{\frac{1}{2}}u_{k_2}, \sum_{k_3 = k_4+O(1)>k_1}(-\triangle)^{\frac{1}{2}}P_{k_1}(u_{k_3}\cdot u_{k_4})\big)\\
& + L((-\triangle)^{\frac{1}{2}}u_{k_2}, (-\triangle)^{\frac{1}{2}}2^{-k_1}L(\nabla_x u_{<k_1-10}, u_{k_1}))
\end{align*}
Then the first term is bounded by 
\begin{align*}
&\big\|\frac12 L((-\triangle)^{\frac{1}{2}}u_{k_2}, \sum_{k_3 = k_4+O(1)>k_1}(-\triangle)^{\frac{1}{2}}P_{k_1}(u_{k_3}\cdot u_{k_4})\big)\big\|_{L_t^1 L_x^2}\\&\lesssim 
2^{k_1}\sum_{k_3 = k_4+O(1)>k_1}\big\|(-\triangle)^{\frac{1}{2}}u_{k_2}\big\|_{L_t^\infty L_x^2}\big\|u_{k_3}\big\|_{L_t^2 L_x^\infty}\big\|u_{k_4}\big\|_{L_t^2 L_x^\infty}\\
&\lesssim \big\|u_{k_2}\big\|_{S_{k_2}}\sum_{k_3 = k_4+O(1)>k_1}2^{k_1 - k_3}\big\|u_{k_3}\big\|_{S_{k_3}}\big\|u_{k_4}\big\|_{S_{k_4}}. 
\end{align*}
This expression can be summed over $k_1$ to give the desired bound. 
Similarly, we get 
\begin{align*}
&\big\| L((-\triangle)^{\frac{1}{2}}u_{k_2}, (-\triangle)^{\frac{1}{2}}2^{-k_1}L(\nabla_x u_{<k_1-10}, u_{k_1}))\big\|_{L_t^1 L_x^2}\\&\lesssim 
\big\|(-\triangle)^{\frac{1}{2}}u_{k_2}\big\|_{L_t^\infty L_x^2}\big\|\nabla_x u_{<k_1-10}\big\|_{L_t^2 L_x^\infty}\big\|u_{k_1}\big\|_{L_t^2 L_x^\infty}\\
&\lesssim \big\|u_{k_2}\big\|_{S_{k_2}}\big\|u_{k_1}\big\|_{S_{k_1}}\big\|u\big\|_{S}, 
\end{align*}
and summation over $k_1<10$ yields the desired bound. This concludes case {\it{(ii)}}, and thereby of \eqref{eq:multilin3}.
\\

The estimates \eqref{eq:multilin4}, \eqref{eq:multilin5} are proved similarly, after passing to the differences. One only needs to make sure to reformulate the terms as in the preceding using  \eqref{eq:orthomicro}, \eqref{eq:orthomicrolocalized}, before passing to the differences. 
\end{proof}

\section{The iteration scheme}

Here we solve \eqref{eq:fracWMfinal}. Specifically, we prove the following 
\begin{thm}\label{thm:intermediate} Let $n\geq 5$. Let $u[0] = (u, u_t): \R^{n}\longrightarrow S^2\times TS^2$ a smooth data pair with $u\cdot u_t = 0$ point wise, and such that $u$ is constant outside of a compact subset of $\R^n$. Also, assume the smallness condition 
\[
\big\|u[0]\big\|_{\dot{B}^{\frac{n}{2},1}_2\times \dot{B}^{\frac{n}{2}-1,1}_2}<\epsilon
\]
where $\epsilon\ll 1$ sufficiently small. Then problem \eqref{eq:fracWMfinal} admits a global smooth solution with these data. 
\end{thm}
\begin{proof} We do this by means of a suitable iteration scheme: first, let $u^{(0)} = p$ where $p\in S^2$ is the limit of the initial data $u|_{t = 0}$ at spatial infinity. Then let $u^{(1)}$ be the Wave Map into $S^2$ with the given data (which is possible since $u_t(0, \cdot)\cdot u(0, \cdot) = 0$ from our assumption), thus solving 
\[
(\partial_t^2 - \triangle) u^{(1)} =  u^{(1)}(\nabla u^{(1)}\cdot\nabla u^{(1)} - \partial_t u^{(1)}\cdot\partial_t u^{(1)})
\]
 It is given by $u^{(1)} = p + \sum_{k\in Z}u_k$, and its existence follows via simple iteration from \eqref{eq:multilin1} and the corresponding difference estimate. Then we define the higher iterates $u^{(j)}, j\geq 2$, via the following iterative scheme: 
 \begin{equation}\label{eq:iterate}\begin{split}
&(\partial_t^2 - \triangle) u^{(j)}\\&= u^{(j)}(\nabla u^{(j)}\cdot\nabla u^{(j)} - \partial_t u^{(j)}\cdot\partial_t u^{(j)})\\
&+\Pi_{u_{\perp}^{(j)}}\big((-\triangle)^{\frac{1}{2}}u^{(j-1)}\big)(u^{(j-1)}\cdot (-\triangle)^{\frac{1}{2}}u^{(j-1)})\\
&+\Pi_{u_{\perp}^{(j)}}\big[u^{(j-1)} \times (-\triangle)^{\frac{1}{2}}(u^{(j-1)} \times (-\triangle)^{\frac{1}{2}}u^{(j-1)}) - u^{(j-1)} \times(u^{(j-1)} \times (-\triangle)u^{(j-1)})\big]
\end{split}\end{equation} 
This equation defines $u^{(j)}$ implicitly, and so to actually compute it, we have to run a sub-iteration
 \begin{equation}\label{eq:iterate1}\begin{split}
&(\partial_t^2 - \triangle) u^{(j,i)}\\&= u^{(j,i-1)}(\nabla u^{(j,i-1)}\cdot\nabla u^{(j,i-1)} - \partial_t u^{(j,i-1)}\cdot\partial_t u^{(j,i-1)})\\
&+\Pi_{u_{\perp}^{(j,i-1)}}\big((-\triangle)^{\frac{1}{2}}u^{(j-1)}\big)(u^{(j-1)}\cdot (-\triangle)^{\frac{1}{2}}u^{(j-1)})\\
&+\Pi_{u_{\perp}^{(j,i-1)}}\big[u^{(j-1)} \times (-\triangle)^{\frac{1}{2}}(u^{(j-1)} \times (-\triangle)^{\frac{1}{2}}u^{(j-1)}) - u^{(j-1)} \times(u^{(j-1)} \times (-\triangle)u^{(j-1)})\big]
\end{split}\end{equation} 
for $i\geq 1$, while $u^{(j,0)}$ is the free wave evolution of the data $u[0]$. 
Then we again have $u^{(j,i)} = p + \sum_k u^{(j,i)}_k$, and in particular each $u^{(j,i)}$ is close to $S^2$ with respect to the $L^\infty$ norm, while convergence with respect to $\big\|\cdot\big\|_S$ follows from Proposition~\ref{prop:multilinear}. We also get higher regularity of each $u^{(j,i)}$ and $u^{(j)}$ by differentiating the equation. 
\\

Our choice of iterative scheme \eqref{eq:iterate} implies 
\[
\Box(u^{(j)}\cdot u^{(j)} - 1) = (u^{(j)}\cdot u^{(j)} - 1)(\nabla u^{(j)}\cdot\nabla u^{(j)} - \partial_t u^{(j)}\cdot\partial_t u^{(j)})
\]
as well as $(u^{(j)}\cdot u^{(j)} - 1)[0] = (0,0)$, which inductively gives that $u^{(j)}$ maps into $S^2$ for all $j$. Finally, convergence of the $u^{(j)}$ with respect to $\|\cdot\|_{S}$ follows again via Proposition~\ref{prop:multilinear}. Differentiating \eqref{eq:iterate} then also gives higher regularity of the limit function $u$. The latter is then easily seen to solve \eqref{eq:fracWMfinal}. For later purposes, we also note that Proposition~\ref{prop:multilinear} in conjunction with the assumptions that $(u-p)\big|_{t = 0}\in C_0^\infty$ and $u_t|_{t = 0} = u\times (-\triangle)^{\frac12}|_{t = 0}$ imply that we have improved control over low frequencies: $u(t, \cdot)\in \dot{H}^{\frac{n}{2}-\frac12}, u_t(t, \cdot)\in \dot{H}^{\frac{n}{2}-\frac32}$ for all $t$. 
\end{proof}

\section{Proof of Theorem~\ref{thm:Main}}

It remains to show that the solution $u(t, x)$ obtained in Theorem~\ref{thm:intermediate} actually solves \eqref{eq:fracwm}. For this introduce the quantity 
\[
X: = u_t - u\times (-\triangle)^{\frac12}u,
\]
as well as the energy type functional 
\[
\tilde{E}(t): = \frac12\int_{\R^n}\big|(-\triangle)^{\frac{n}{4} - \frac{3}{4}}X(t,\cdot)\big|^2\,dx.
\]
Note that we have $\nabla_{t,x}u\in \dot{H}^{\frac{n}{2}-\frac32}$ as observed previously, and hence $\tilde{E}(t)$ is well defined and also continuously differentiable (on account of the higher regularity properties of $u$). Retracing the steps that led to the final wave equation \eqref{eq:fracWMfinal}, we deduce 
\[
\partial_tX = -X\times  (-\triangle)^{\frac12}u  - u\times(-\triangle)^{\frac12}X - u\big(X\cdot(u\times  (-\triangle)^{\frac12}u + u_t)\big),
\]
and so we deduce 
\begin{align*}
\frac{d}{dt}\tilde{E}(t) =& -\int_{\R^n}(-\triangle)^{\frac{n}{4} - \frac{3}{4}}\big(X\times  (-\triangle)^{\frac12}u +  u\times(-\triangle)^{\frac12}X\big)\cdot (-\triangle)^{\frac{n}{4} - \frac{3}{4}}X\,dx\\
& - \int_{\R^n}(-\triangle)^{\frac{n}{4} - \frac{3}{4}}\big(u\big(X\cdot(u\times  (-\triangle)^{\frac12}u + u_t)\big)\big)\cdot(-\triangle)^{\frac{n}{4} - \frac{3}{4}}X\,dx.\\
\end{align*}
Then we note that\footnote{Here $\frac{2n}{n-5}$ gets replaced by $\infty$ of $n=5$.} 
\begin{align*}
&\big\|(-\triangle)^{\frac{n}{4} - \frac{3}{4}}\big(X\times  (-\triangle)^{\frac12}u +  u\times(-\triangle)^{\frac12}X\big) - u\times(-\triangle)^{\frac{n}{4} - \frac{1}{4}}X \big\|_{L_x^2}\\&\lesssim \big\|(-\triangle)^{\frac{n}{4} - \frac{3}{4}}X\big\|_{L_x^2} \big\| (-\triangle)^{\frac12}u\big\|_{L_x^\infty} + \big\|X\big\|_{L_x^{\frac{2n}{3}}} \big\|(-\triangle)^{\frac{n}{4} - \frac{1}{4}}u\big\|_{L_x^{\frac{2n}{n-3}}}\\
& + \big\|(-\triangle)^{\frac{n}{4} - \frac{3}{4}}u\big\|_{L_x^{\frac{2n}{n-5}}} \big\|(-\triangle)^{\frac12}X\big\|_{L_x^{\frac{2n}{5}}}\lesssim_u \big\|(-\triangle)^{\frac{n}{4} - \frac{3}{4}}X\big\|_{L_x^2}^2
\end{align*}
on account of Sobolev's embedding and higher regularity of $u$, and further, we observe that 
\begin{align*}
&\int_{\R^n}\big(u\times(-\triangle)^{\frac{n}{4} - \frac{1}{4}}X\big)\cdot (-\triangle)^{\frac{n}{4} - \frac{3}{4}}X\,dx\\
& = \int_{\R^n}(-\triangle)^{\frac14}\big(u\times(-\triangle)^{\frac{n}{4} - \frac{2}{4}}X\big)\cdot (-\triangle)^{\frac{n}{4} - \frac{3}{4}}X\,dx + O(\big\|(-\triangle)^{\frac{n}{4} - \frac{3}{4}}X\big\|_{L_x^2}^2\big\|\nabla_x u\big\|_{L_x^\infty})\\
& =  O(\big\|(-\triangle)^{\frac{n}{4} - \frac{3}{4}}X\big\|_{L_x^2}^2\big\|\nabla_x u\big\|_{L_x^\infty}).
\end{align*}
Similarly, we infer 
\begin{align*}
&\big| \int_{\R^n}(-\triangle)^{\frac{n}{4} - \frac{3}{4}}\big(u\big(X\cdot(u\times  (-\triangle)^{\frac12}u + u_t)\big)\big)\cdot(-\triangle)^{\frac{n}{4} - \frac{3}{4}}X\,dx\big|\\&\lesssim_u \big\|(-\triangle)^{\frac{n}{4} - \frac{3}{4}}X\big\|_{L_x^2}^2.
\end{align*}
But then the preceding implies that 
\[
\frac{d}{dt}\tilde{E}(t)\leq C(u)\tilde{E}(t)
\]
and furthermore $\tilde{E}(0) = 0$, which implies $\tilde{E}(t) = 0$ throughout. It follows that $X = 0$ identically, which completes the proof of Theorem~\ref{thm:Main}.

\centerline{\scshape Joachim Krieger }
\medskip
{\footnotesize
 \centerline{B\^{a}timent des Math\'ematiques, EPFL}
\centerline{Station 8, 
CH-1015 Lausanne, 
  Switzerland}
  \centerline{\email{joachim.krieger@epfl.ch}}
} 

\medskip

\centerline{\scshape Yannick Sire}
\medskip
{\footnotesize
 \centerline{Department of Mathematics, Johns Hopkins University}
\centerline{3400 N. Charles Street, Baltimore, MD 21218 U.S.A.}
\centerline{\email{sire@math.jhu.edu}
}
} 

\end{document}